\documentclass[a4paper,12pt,draft]{amsart}
\usepackage[margin=30mm]{geometry}
\usepackage{amssymb, amsmath, amsthm, amscd}

\usepackage[T1]{fontenc}
\usepackage{eucal,mathrsfs,dsfont}%gives nicer set-names. Use \mathds instead of \mathds
\usepackage{color}%cyan,red;magenta,green;yellow,blue
\renewcommand{\le}{\leqslant}
\renewcommand{\ge}{\geqslant}

\definecolor{mno}{rgb}{0.5,0.1,0.5}

\newcommand{\R}{\mathds R}

\newcommand{\Pp}{\mathds P}

\newcommand{\I}{\mathds 1}

\newcommand{\Z}{\mathds Z}

\newtheorem{theorem}{Theorem}[section]
\newtheorem{lemma}[theorem]{Lemma}
\newtheorem{proposition}[theorem]{Proposition}
\newtheorem{corollary}[theorem]{Corollary}

\theoremstyle{definition}
\newtheorem{definition}[theorem]{Definition}

\newtheorem{remark}[theorem]{Remark}

\begin{document}
\allowdisplaybreaks
\title[Weighted Fractional Heat Semigroups] {\bfseries
Compactness and Density Estimates for Weighted Fractional Heat Semigroups}

\author{Jian Wang}
  \thanks{\emph{J.\ Wang:}
   College of Mathematics and Informatics $\&$ Fujian Provincial
Key Laboratory of Mathematical Analysis and its Applications
(FJKLMAA), Fujian Normal University, 350007 Fuzhou, P.R. China. \texttt{jianwang@fjnu.edu.cn}}

\date{}

\maketitle

\begin{abstract} We prove that the operator $L_0=-(1+|x|)^\beta(-\Delta)^{\alpha/2}$ with $\alpha\in(0,2)$, $d>\alpha$ and $\beta\ge0$ generates a compact semigroup or resolvent on $L^2(\R^d;(1+|x|)^{-\beta}\,dx)$, if and only if $\beta>\alpha$. When $\beta>\alpha$, we obtain two-sided asymptotic estimates for high order eigenvalues, and sharp bounds for the corresponding heat kernel.
\medskip

\medskip

\end{abstract}
\allowdisplaybreaks
\section{Introduction}
We consider analytic properties of the weighted fractional Laplacian operator \begin{equation}\label{e:ger}L_0=-(1+|x|)^\beta(-\Delta)^{\alpha/2}\end{equation} on $\R^d$, where $\alpha\in(0,2)$, $d>\alpha$ and $\beta\ge0$.  Obviously the operator $L_0$ is symmetric with respect to the measure $\mu(dx)=(1+|x|)^{-\beta}\,dx.$ We will prove that the operator $L_0$ generates a compact semigroup or resolvent on $L^2(\R^d;\mu)$ if and only if $\beta>\alpha$. In this case, the spectrum of $-L_0$ consists of a sequence of positive eigenvalues $\lambda_n$ diverging to $\infty$ as $n\to \infty$, and the associated semigroup $(P_t)_{t\ge0}$ admits a density function $p_\mu(t,x,y)$ with respect to $\mu$; namely,
$$P_tf(x)=\int_{\R^d} p_\mu(t,x,y)f(y)\,\mu(dy),\quad f\in L^2(\R^d;\mu).$$ Furthermore, we establish the following asymptotic estimates for high order eigenvalues
$$0<\liminf_{n\to\infty}\frac{\lambda_n}{n^{\alpha/d}}\le \limsup_{n\to\infty}\frac{\lambda_n}{n^{\alpha/d}}<\infty,$$ and obtain that there are constants $0<c_1\le c_2<\infty$ such that for all $t\in (0,1)$,
$$c_1t^{-d/\alpha}\le \sup_{x,y\in\R^d}\frac{ p_\mu(t,x,y)}{\phi^{(\beta/(2\alpha))\wedge1}(x)\phi^{(\beta/(2\alpha))\wedge1}(y)}\le c_2t^{-d/\alpha},$$ where $\phi$ is the eigenfunction (also called the ground state in the literature) corresponding to the first eigenvalue.

When $\alpha=2$, $L_0$ becomes a Laplacian operator with unbounded diffusion coefficient; that is,
 $L_0=(1+|x|)^\beta\Delta.$ In this setting, the solvability of the elliptic and the parabolic problems corresponding to the operator $L_0$ have been already investigated in \cite{MS12-2, MPW}. In particular, according to \cite[Proposition 2.2]{MS12-2}, the associated semigroup and resolvent are compact if and only if $\beta>2$. Later heat kernel (i.e., density function of the semigroup) estimates for the same operator have been obtained in \cite{MS12, MS13}. To the best of our knowledge, the corresponding results for $\alpha\in(0,2)$ are not available. The aim of this paper is to fill the gap. Due to the non-local property of fractional Laplacian operator, the approach of \cite{MS12-2, MPW} is not efficient in the present setting. To obtain sharp criteria for the compactness, we will make use of super Poincar\'e inequalities for non-local Dirichlet forms. On the other hand, similar to \cite{MS12, MS13}, upper bound estimates for heat kernels are derived by using weighted Nash inequalities. These along with recent developments on Dirichlet heat kernels for fractional Laplacian operators (see \cite{KK} and the references therein) in turn yield  asymptotic estimates for the associated high order eigenvalues.

We should mention that there is much interest towards fractional Laplacian operator with unbounded coefficients in the theory of stochastic processes. For $\alpha\in(0,2)$, let $X=(X_t)_{t\ge0}$ be a rotationally symmetric $\alpha$-stable process on $\R^d$. It is well known that
the infinitesimal generator of $X$ is the fractional Laplacian $-(-\Delta)^{\alpha/2}$. Next, we consider a time-change of $X$. Let $\mu(dx)=(1+|x|)^{-\beta}\,dx$ and $A_t=\int_0^t
(1+|X_s|)^{-\beta}\, ds$ for any $t\ge0$, which is a positive continuous additive functional of
$X$. For any $t>0$, define $\tau_t=\inf\{s>0:A_s>t\}$ and set $X^\mu_t=X_{\tau_t}$. It is known that $X^\mu=(X_t^\mu)_{t\ge0}$ is a $\mu$-symmetric Markov process on $\R^d$, see \cite[Theorem 5.2.2]{CF} or \cite[Theorem 6.2.1]{FOT}. We can easily verify that the infinitesimal generator of $X^\mu$ is just the operator $L_0$ given by \eqref{e:ger}. The ergodicity for one-dimensional symmetric Markov jump
processes associated with time-changed symmetric $\alpha$-stable processes has been considered in \cite{CW13} by using Poincar\'e type inequalities, and the H\"{o}lder continuity of semigroups was obtained in \cite{LW} by using the coupling method. As shown in \cite[Example 2.2]{KP}, the time-changed symmetric $\alpha$-stable process is a typical example in  study of the ergodicity for jump-diffusion processes.

According to results on (intrinsic) super Poincar\'e inequalities for the essential spectrum introduced in \cite{Wang01, Wang02} (see also \cite[Chapter 3]{WBook} and references therein), in order to make a non-compact semigroup compact one may enlarge the associated Dirichlet form. One way to enlarge the Dirichlet form is by adding a potential term, which leads to the study of Schr\"{o}dinger semigroups; the other way is to multiplying a weight to the energy (square field) as done in the present paper, and it corresponds to a time change as mentioned above. The former has been done very recently in \cite{JW} for non-local Schr\"{o}dinger operators. Theoretically, by using a ground state transform (see e.g.\ \cite[Chapter 10]{Gr}), these two ways can be transferred into each other. However, we should emphasize that these two ways may yield completely different analytic properties. For example, asymptotic estimates for high order eigenvalues of fractional Laplacian Schr\"{o}dinger operators will heavily depend on the potential term, see \cite[Example 2.2]{JW}, while those for weighted fractional  Laplacian operators can be independent of the weighted function, see Theorem \ref{T:main}(i) below.

The reminder of this paper is arranged as follows. In the next section, we are concerned about criteria for the compactness of semigroup and resolvent corresponding to the operator $L=-W(x)(-\Delta)^{\alpha/2}$ by applying the so called instric super Poincar\'e inequalities first introduced in \cite{Wang02} . General results for the compactness of the semigroup corresponding to time-changed Dirichlet forms  are given in terms of (local) super Poincar\'e inequalities plus Orlicz-Sobolev inequalities or Hardy type inequalities, which are interesting of their own. Section \ref{eigen} is devoted to asymptotic estimates for high order eigenvalues, which are based on explicit and sharp estimates for the associated heat kernels.
\section{Compactness of semigroups associated with time-changed Dirichlet forms}
\subsection{Super Poincar\'e inequalities for time-changed Dirichlet forms}
Let $(E,\rho)$ be a Polish space with Borel $\sigma$-algebra $\mathscr{F}$, and $\nu$ $\sigma$-finite measure on $E$. Let $(\mathscr{E}, \mathscr{D}(\mathscr{E}))$ be a symmetric Dirichlet form on $L^2(\nu)$ with generator $(L,\mathscr{D}(L))$. Let $\kappa$ be a strictly positive measurable function on $E$ such that it is bounded from below. Define $L_\kappa f=\kappa\cdot Lf$ and $\nu_\kappa(dx)=\kappa(x)^{-1}\,\nu(dx).$ Then,
the operator $L_\kappa$ is symmetric with respect to the measure $\nu_\kappa$. Denote by $\langle \cdot, \cdot\rangle_{L^2(\nu)}$ and $\langle \cdot, \cdot\rangle_{L^2(\nu_\kappa)}$ the inner product in $L^2(\nu)$ and $L^2(\nu_\kappa)$, respectively. The bilinear form corresponding to $-L_\kappa$ on $L^2(\nu_\kappa)$ is given by
\begin{align*}\mathscr{E}_\kappa(f,g):=&-\langle L_\kappa f, g\rangle_{L^2(\nu_\kappa)}=-\int_E g(x) L_\kappa f(x)\,\nu_\kappa(dx)\\
=&-\int_E g(x)\cdot \kappa(x)\cdot Lf(x)\cdot \kappa(x)^{-1}\,\nu(dx)=-\langle L f, g\rangle_{L^2(\nu)}=\mathscr{E}(f,g),\end{align*} where $f,g\in L^2(\nu_\kappa)\cap \mathscr{D}(L)=:\mathscr{D}(L_\kappa)$. That is, $\mathscr{E}_\kappa$ enjoys the same formula as $\mathscr{E}$ on $\mathscr{D}(L_\kappa)\times \mathscr{D}(L_\kappa)$. The reader can refer to \cite[Section 5.5.2]{CF} for more details on the form $\mathscr{E}_\kappa$.

In the following, we fix $o\in E$. Assume that the following local super Poincar\'e inequality holds for $(\mathscr{E}, \mathscr{D}(\mathscr{E}))$:
\begin{equation}\label{e:local} \nu (f^2\I_{B(o,r)})\le s \mathscr{E}(f,f)+\beta(s,r)\nu (|f|\I_{B(o,r)})^2,\quad s,r>0,  f\in \mathscr{D}(\mathscr{E}),\end{equation}
where $\beta(s,r)$ satisfies that $s\mapsto \beta(s,r)$ is decreasing for any fixed $r>0$, and $r\mapsto \beta(s,r)$ is increasing for any fixed $s>0$.
\begin{proposition}\label{P:sup} Suppose that \eqref{e:local} holds true.
\begin{itemize}
\item[(i)] Assume that there exist a Young function $N$ and a constant $c>0$ such that for all $f\in \mathscr{D}(\mathscr{E})$,
\begin{equation}\label{T:sup1}\|f^2\|_N \le c\mathscr{E}(f,f)\end{equation} and $\kappa^{-1}\in L_{\hat N}(\nu),$ where $\hat N (r)=\sup_{s\ge0}\{sr-N(s)\}$ for all $r\ge0$. Then, for all $f\in \mathscr{D}(\mathscr{E})\cap \mathscr{D}(L_\kappa)$ and $s,r>0$, \begin{equation}\label{T:sup1-1}\begin{split}\nu_\kappa(f^2)\le& \left[s\left(\sup_{B(o,r)}\kappa^{-1}\right)+c  \|\kappa^{-1}\I_{B(o,r)^c}\|_{\hat N}\right]\mathscr{E}(f,f)\\
&+ \beta(s,r)\left(\sup_{B(o,r)}\kappa^{-1}\right)\nu(|f|\I_{B(o,r)})^2.\end{split}\end{equation}
\item[(ii)] If there exist a non-negative measurable function $\kappa_0$ and a constant $c>0$ such that
\begin{equation}\label{T:sup2}\int_E f^2(x)\kappa_0(x)^{-1}\,\nu(dx)\le c \mathscr{E}(f,f), \quad f\in \mathscr{D}(\mathscr{E})\end{equation} and
$$\lim_{d(x,o)\to\infty}  \frac{\kappa_0(x)}{\kappa(x)}=0,$$ then for all $f\in \mathscr{D}(\mathscr{E})\cap \mathscr{D}(L_\kappa)$ and $s,r>0$, \begin{equation}\label{T:sup2-1}\begin{split}\nu_\kappa(f^2)\le& \left[s\left(\sup_{B(o,r)}\kappa^{-1}\right)+c\left(\sup_{x\in B(o,r)^c}\frac{\kappa_0(x)}{\kappa(x)} \right)\right]\mathscr{E}(f,f)\\
&+ \beta(s,r)\left(\sup_{B(o,r)}\kappa^{-1}\right) \nu(|f|\I_{B(o,r)})^2.\end{split}\end{equation}
\end{itemize}\end{proposition}
Recall that a function $N:[0,\infty)\to [0,\infty)$ is called a Young function, if $N(0)=0$ and $N$ is increasing and continuous. The Orlicz norm induced by the Young function $N$ is given by
$$\|f\|_{N}:=\inf\bigg\{r>0: \int_{E} N\left(\frac{|f(x)|} r\right)\, \nu(dx) \le 1\bigg\},$$ where $\inf \emptyset =\infty$ by convention. Denote by $L_N(\nu)=\{f:\|f\|_N<\infty\}$. In the literature, \eqref{T:sup1} is called the Orlicz-Sobolev inequality, e.g.,\ see \cite{Wang05}.
\eqref{T:sup2} comes from the so-called Hardy type inequality, and the reader can refer to \cite[Theorem 2]{BDK} in general setting.

\begin{proof}[Proof of Proposition $\ref{P:sup}$] For any $f\in \mathscr{D}(\mathscr{E})\cap \mathscr{D}(L_\kappa)$ and $r>0$,
$$\nu_\kappa(f^2)=\nu(f^2\kappa^{-1})=\nu(f^2\kappa^{-1}\I_{B(o,r)})+\nu(f^2\kappa^{-1}\I_{B(o,r)^c}).$$ By \eqref{e:local}, it holds that for all $s>0$,
\begin{equation}\label{e:f1}\begin{split}\nu(f^2\kappa^{-1}\I_{B(o,r)})
&\le \left(\sup_{B(o,r)}\kappa^{-1}\right) \nu(f^2\I_{B(o,r)})\\
&\le\left(\sup_{B(o,r)}\kappa^{-1}\right)\left[ s \mathscr{E}(f,f)+\beta(s,r)\nu(|f|\I_{B(o,r)})^2\right] \\
&= s\left(\sup_{B(o,r)}\kappa^{-1}\right)\mathscr{E}(f,f)+ \beta(s,r)\left(\sup_{B(o,r)}\kappa^{-1}\right)\nu(|f|\I_{B(o,r)})^2.\end{split}\end{equation}

(1) If \eqref{T:sup1} holds, then, by the generalized H\"{o}lder inequality (see \cite[Chapter III, Proposition 3.1]{RR}),
$$\nu(f^2\kappa^{-1}\I_{B(o,r)^c})\le \|f^2\|_{N}\|\kappa^{-1}\I_{B(o,r)^c}\|_{\hat N}\le c \mathscr{E}(f,f) \|\kappa^{-1}\I_{B(o,r)^c}\|_{\hat N}.$$ Combining this with \eqref{e:f1}, we obtain \eqref{T:sup1-1}.

(2) If \eqref{T:sup2} is satisfied, then
\begin{align*} \nu(f^2\kappa^{-1}\I_{B(o,r)^c})=&\int_{B(o,r)^c}\frac{f^2(x)}{\kappa_0(x)}\frac{\kappa_0(x)}{\kappa(x)}\,\nu(dx)\\
\le&\left(\sup_{x\in B(o,r)^c}\frac{\kappa_0(x)}{\kappa(x)} \right)\int_{B(o,r)^c}{f^2(x)}{\kappa_0(x)^{-1}}\,\nu(dx)\\
\le&c\left(\sup_{x\in B(o,r)^c}\frac{\kappa_0(x)}{\kappa(x)} \right)\mathscr{E}(f,f).\end{align*} This along with  \eqref{e:f1} yields \eqref{T:sup2-1}.
\end{proof}
\subsection{Criteria for the compactness of semigroup  generated by weighted fractional Laplacian operators}
In this subsection, we consider the operator $L=-W(x)(-\Delta)^{\alpha/2}$, where $W(x)\ge 1$ for all $x\in \R^d$. Recall that the Dirichlet form $(D,\mathscr{D}(D))$ of the symmetric $\alpha$-stable process $X=(X_t)_{t\ge0}$ on $L^2(\R^d;dx)$ is given by
\begin{align*}D(f,g)=&\frac{1}{2}\int_{\R^d\times \R^d}(f(x)-f(y))(g(x)-g(y))\frac{c_{d,\alpha}}{|x-y|^{d+\alpha}}\,dx\,dy,\\
\mathscr{D}(D)=&\big\{f\in L^2(\R^d;dx): D(f,f)<\infty\}.\end{align*} We know from \cite[Section 2.2.2]{CF} that $(D,\mathscr{D}(D))$ is a regular Dirichlet form, and $C_c^\infty(\R^d)$ is a core for $(D,\mathscr{D}(D))$. Let $\mathscr{D}_e(D)$ be the extended Dirichlet space of $(D,\mathscr{D}(D))$ for the process $X$. When $d>\alpha$, the process $X$ is transient, and by \cite[(2.2.18)]{CF}
$$\mathscr{D}_e(D)=\big\{ f\in L^1_{loc}(\R^d)\cap \mathscr{S}': D(f,f)<\infty \big\},$$  where $\mathscr{S}'$ denotes the space of tempered distributions. When $d\le \alpha$, the process $X$ is recurrent, and it further follows from \cite[(6.5.4)]{CF} that
$$\mathscr{D}_e(D)=\big\{ f: \text{Borel measurable with } |f|<\infty \text{ a.e. and } D(f,f)<\infty \big\}.$$
Let $\mu(dx)=W(x)^{-1}\,dx$ and $A_t=\int_0^t
W(X_s)^{-1}\, ds$ for any $t\ge0$, which is a positive continuous additive functional of
$X$. For any $t>0$, define $\tau_t=\inf\{s>0:A_s>t\}$ and set $X^\mu_t=X_{\tau_t}$. It follows from \cite[Corollary 3.3.6, Theorem 5.2.2 and Theorem 5.2.8]{CF} that  $X^\mu=(X_t^\mu)_{t\ge0}$ is a $\mu$-symmetric Markov process on $\R^d$, whose associated Dirichlet form $(D, \mathscr{D}(D^\mu))$ is regular on $L^2(\R^d;\mu)$ having
$C_c^\infty(\R^d)$ as its core. In view of \cite[Corollary 5.2.12]{CF}, $\mathscr{D}(D^\mu)= \mathscr{D}_e(D)\cap L^2(\R^d;\mu).$ By abusing the notation  a little bit, we also denote by $L^2$-generator of $X$ by $-(-\Delta)^{\alpha/2}$. Then, the $L^2$-generator of $X^\mu$ is the operator $L$.

Next, we set $\nu$ to be the Lebesgue measure, $\kappa(x)=W(x)$ and $\mu(dx)=\nu_\kappa(dx)=\kappa(x)^{-1}\,dx= W(x)^{-1}\,dx$ in notations of the previous subsection.
\begin{theorem}\label{T:com1}Assume that $d>\alpha$. Then, we have the following two statements.
\begin{itemize}
\item[(i)]Let $\Psi_1(r)=\sup_{|x|\ge r}\frac{|x|^\alpha}{W(x)}$, $\Psi_2(r)=\sup_{|x|\le r}W(x)^2$ and $$\beta_0(r)=c_1(1+r^{-d/\alpha})\Psi_2(\Psi_1^{-1}(r\wedge c_2))$$ for any $r,c_1,c_2>0$, where $\Psi_1^{-1}(r):=\inf\{t>0:\Psi_1(t)\le r\}.$ If $\lim\limits_{r\to\infty}\Psi_1(r)=0$ and  $$\int_s^\infty\frac{ \beta_0^{-1}(r)}{r}\,dr<\infty$$ for large $s>0$, then, the operator $L$ generates  a compact semigroup or resolvent on $L^2(\R^d;\mu)$, where $\beta_0^{-1}(r):=\inf\{t>0:\beta(t)\le r\}.$
\item[(ii)] If $$\sup_{x\in \R^d}\frac{W(x)}{(1+|x|)^\alpha}<\infty,$$ then the semigroup or resolvent generated by $L$ is not compact.
\end{itemize} \end{theorem}
\begin{proof}
(1) We assume that assumptions in the statement (i) hold. The local Nash inequality for the fractional Laplacian gives us
that there exists a constant $c_0>0$ such that for all $r,s>0$ and $f\in C_c^\infty(\R^d)$,
$$\int_{\{|x|\le r\}} f^2(x)\,dx\le s D(f,f)+c_0(1+s^{-d/\alpha})\left(\int_{\{|x|\le r\}}|f|(x)\,dx\right)^2.$$ See \cite[(2.16) in the proof of Lemma 2.1]{CW14}. Thus, the local super Poincar\'e inequality \eqref{e:local} holds with $\beta(s,r)=c_0(1+s^{-d/\alpha})$ for all $s,r>0$, which is independent of $r$.
On the other hand, the Hardy inequality for the
fractional Laplacian on $\R^d$ with $d>\alpha$ (for example, see \cite[Theorem 1.1]{D04} or \cite[Proposition 5]{BDK}) says that \eqref{T:sup2} holds with $\kappa_0(x)=|x|^\alpha$, which along with \eqref{T:sup2-1} and the definition of $\Psi_1(r)$ in turn yields that for all $s>0$, $r>0$ and $f\in C_c^\infty(\R^d)$,
\begin{equation}\label{e:sup1}\mu(f^2)\le (s+c_1\Psi_1(r))D(f,f)+ c_0(1+s^{-d/\alpha})\nu(|f|\I_{B(0,r)})^2.\end{equation}

By \eqref{e:sup1} and the definition of $\Psi_2(r)$,
$$\mu(f^2)\le (s+c_1\Psi_1(r))D(f,f)+ c_0(1+s^{-d/\alpha})\Psi_2(r)\mu(|f|)^2.$$ Then,
taking $r=\Psi_1^{-1}(s\wedge c_2)$ for all $s>0$ and some $c_2>0$ in the inequality above (since $\lim_{r\to \infty}\Psi_1(r)=0$, one may choose $c_2$ small enough such that $0< r=\Psi_1^{-1}(s\wedge c_2)<\infty$), we can obtain
that for all $s>0$ and $f\in C_c^\infty(\R^d)$,
\begin{equation}\label{e:sup2}\mu(f^2)\le s
D(f,f)+\beta_0(s)\mu(|f|)^2,\end{equation} where $\beta_0(s)$ (possibly with different constants $c_1$ and $c_2$) is defined in the statement (i). Also due to the assumption that $\lim_{r\to \infty}\Psi_1(r)=0$, $\beta_0(s)<\infty$ for $s>0$ small enough. Again, by the assumption in (i),
$$\int_s^\infty\frac{ \beta_0^{-1}(r)}{r}\,dr<\infty$$ for large $s>0$, then $P_t=e^{tL}$ is ultracontractive, i.e.,\ $\|P_t\|_{L^1(\R^d;\mu)\to L^\infty(\R^d;\mu)}<\infty$ for all $t>0$, so that
the semigroup $P_t$ has a density $p^\mu(t,x,y)$ with respect to $\mu$ for all $t>0$, see \cite[Theorem 3.3.14]{WBook}. Moreover, by \cite[Theorem 3.1]{BBCK}, there is an $(D,\mathscr{D}(D^\mu))$-nest $(F_k)_{k\ge1}$ of compact sets (see \cite[p.\ 69]{FOT} for the definition of nest set) so that $\cup_{k=1}^\infty F_k=\R^d$
and that for every $t>0$ and $y\in \R^d$, $x\mapsto p_\mu(t,x,y)$ is continuous on each $F_k$.

Following the approach of \eqref{e:sup2}, we can also obtain that for all $s>0$ and $f\in C_c^\infty(\R^d)$,
\begin{equation}\label{e:sup3}\mu(f^2)\le s
D(f,f)+\beta_1(s)\mu(|f|\psi)^2,\end{equation} where $\psi(x)=e^{-|x|}\in L^2(\R^d;\mu)$ and $$\beta_1(s)=c_4\beta_0(s)\exp(2\Psi_1^{-1}(s\wedge c_5))<\infty,$$ also thanks to the assumption that $\lim_{r\to0}\Psi_1(r)=0$. Then, combining \eqref{e:sup2} and \eqref{e:sup3} with \cite[Theorem
3.1.7, Theorem 3.2.2 and Theorem 0.3.9]{WBook}, we know that the operator $L$ generates a compact semigroup or resolvent on $L^2(\R^d;\mu)$.

We also mention that the fractional Sobolev inequality on $\R^d$ with $d>\alpha$ yields that \eqref{T:sup1} holds
with $N(r)=r^{d/(d-\alpha)}$. Using this and Proposition \ref{P:sup}(i), one can obtain the desired assertion by following the argument above. The details are left to readers.

(2) We now assume that the assumption in the statement (ii) is satisfied.  Suppose that the operator $L=-W(x)(-\Delta)^{\alpha/2}$ can generate a compact semigroup on $L^2(\R^d;\mu)$. Then, according to \cite[Theorem
3.1.7, Theorem 3.2.2 and Theorem 0.3.9]{WBook} again, for any non-negative measurable function $\psi\ge0$ such that
$\mu(\psi^2)<\infty$, there exists $\beta:(0,\infty)\to(0,\infty)$
such that for all $f\in \mathscr{D}(D^\mu)$ and $r>0$,
\begin{equation}\label{e:ins1}\mu(f^2)\le r D(f,f)+\beta(r)\mu(|f|\psi)^2.\end{equation} In the following, we take $\psi(x)=e^{-|x|}$, and
$f_l\in C_c^\infty(\R^d)$ for all $l\ge1$ such that
\begin{equation*}
f_l(x)
\begin{cases}
=1, &\text{if}\ |x|\le l,\\
\in[0,1], &\text{if}\ l\le |x|\le 2l,\\
=0,&\text{if}\ |x|\ge 2l.
\end{cases}
\end{equation*}
We see from the assumption in (ii) that
$$\mu(f_l^2)\ge  c_0\int_{\{|x|\le l\}}\frac{1}{(1+|x|)^\alpha}\,dx\ge c_1 l^{d-\alpha}.$$ On the other hand,
\begin{align*}D(f_l,f_l)&\le c_2\int_{\{|x|\le 2l\}}\int \frac{(f_l(x)-f_l(y))^2}{|x-y|^{d+\alpha}}\,dy\,dx\\
&\le \frac{c_3}{l^2}\int_{\{|x|\le 2l\}}\int_{\{|y|\le 4l\}} \frac{|x-y|^2}{|x-y|^{d+\alpha}}\,dy\,dx\\
&\quad+ c_3\int_{\{|x|\le 2l\}}\int_{\{|y|\ge 4l\}}\frac{1}{|x-y|^{d+\alpha}} \,dy\,dx\\
&\le \frac{c_4}{l^2}\int_{\{|x|\le 2l\}}\int_{\{|x-y|\le 6l\}} \frac{|x-y|^2}{|x-y|^{d+\alpha}}\,dy\,dx\\
&\quad+ c_4\int_{\{|x|\le 2l\}}\int_{\{|x-y|\ge 2l\}}\frac{1}{|x-y|^{d+\alpha}} \,dy\,dx\\
&\le c_5 l^{d-\alpha},\end{align*} and
$$\mu(f_l\psi)^2\le \left(\int_{\R^d}\frac{1}{W(x)} e^{-|x|}\,dx\right)^2\le \left(\int_{\R^d}e^{-|x|}\,dx\right)^2\le c_6.$$
Applying all the estimates into \eqref{e:ins1}, we find that for all $r>0$ and $l\ge 1$,
$$1\le c_7 r +c_7\beta(r)l^{\alpha-d}.$$
Letting $l\to\infty$, we find that $1\le c_7r$ for all $r>0$, thanks to the fact that $d>\alpha$. This is a contradiction. Therefore, under the assumption of the statement (ii) the operator $L=-W(x)(-\Delta)^{\alpha/2}$ can not generate a compact semigroup on $L^2(\R^d;\mu)$ The proof is complete.
\end{proof}

As a direct consequence of Theorem \ref{T:com1}, we immediately have the following

\begin{corollary}\label{T:com}Assume the $d>\alpha$. Then, the operator $L_0=-(1+|x|)^\beta(-\Delta)^{\alpha/2}$ with $\beta\ge0$ can generates a compact semigroup on $L^2(\R^d;(1+|x|)^{-\beta}\,dx)$ if and only if $\beta>\alpha$. \end{corollary}

The corollary below further illustrate the power of Theorem \ref{T:com1}.
\begin{corollary}\label{E:1} Assume that $d>\alpha$ and the weighted function $W(x)$ satisfies that
\begin{equation}\label{E:11}c_1(1+|x|)^\beta\le W(x)\le c_2\exp[c_3(1+|x|)^{(\beta-\alpha)/\delta}],\quad x\in \R^d\end{equation}for some constants $\beta>\alpha$, $\delta>1$ and $c_1,c_2,c_3>0$. Then, the operator $L=-W(x)(-\Delta)^{\alpha/2}$ can generates a compact semigroup on $L^2(\R^d;\mu)$.  \end{corollary}
\begin{proof} It follows from \eqref{E:11} that $\Psi_1(r)\le c_1 (1+r)^{\alpha-\beta}$ and $$\Psi_2(r)\le c_2\exp[c_3(1+r)^{(\beta-\alpha)/\delta}].$$ Then,
 $$\beta_0(r)\le c_4\exp(c_5(1+r^{-1/\delta})).$$ This along with Theorem \ref{T:com1}(i) yields the desired assertion.\end{proof}

\section{Heat Kernel Estimates and Asymptotic Estimates for High Order Eigenvalues}\label{eigen}
\subsection{Heat kernel estimates and asymptotic estimates for high order eigenvalues related to $L=-W(x)(-\Delta)^{\alpha/2}$}
In this subsection, we always let \emph{$d>\alpha$, and $L=-W(x)(-\Delta)^{\alpha/2}$, where $W(x)$ is a weighted function satisfying \eqref{E:11} with $\beta>\alpha$ and $\delta>1$}.  Then, according to Corollary \ref{E:1}, the operator $L$ generates a compact semigroup on $L^2(\R^d;\mu)$, where $\mu(dx)=W(x)^{-1}\,dx$.
Let us denote by (counting multiplicities) $0<\lambda_1\le \lambda_2\le \cdots\le \cdots $ the eigenvalues of $-L$, and by $\phi_n$ the corresponding eigenfunctions, which we assume to be normalized in $L^2(\R^d;\mu)$. We write $\phi_1$ as $\phi$ for simplicity. The aim of this section is to obtain behavior of high order eigenvalues estimates $\lambda_n$.
First, we have the following property for $\phi$.
\begin{lemma}\label{L:re} The first eigenvalue $\lambda_1$ is one-dimensional, and the associated eigenfunction $\phi$ has a strictly positive and bounded version such that for all compact sets $A$, $\inf_{x\in A}\phi(x)>0.$\end{lemma}
\begin{proof}For any open set $A$ of $\R^d$, let $\tau_{A}^{X^\mu}$ be the first exit time of the process $X^\mu$ from $A$, i.e., $\tau_A^{X^\mu}=\inf\{t>0:X^\mu_t\notin A\}.$ For any $x_0\in \R^d$ and $A\subset\R^d$, we choose $n:=n(x_0,A)\ge1$ large enough such that $A\subset B(x_0,n)$. Then for $t>0$,
\begin{align*}\Pp^{x_0}(X_t^\mu\in A)&\ge \Pp^{x_0}(X_t^\mu\in A, \tau_{B(x_0,n)}^{X^\mu}>t)=\hat\Pp^{x_0}(\hat X_t\in A,\tau_{B(x_0,n)}^{\hat X}>t),\end{align*} where $(\hat X_t)_{t\ge0}$ is a symmetric Hunt process on $\R^d$ with the operator $\hat L=-\hat a (-\Delta)^{\alpha/2}$, and $0<a_1(x_0,n)=:a_1\le \hat a(z)\le a_2:=a_2(x_0,n)<\infty$ for all $z\in \R^d$ such that $\hat a(z)=W(z)$ for all $z\in B(x_0,n)$. Clearly, $(\hat X_t)_{t\ge0}$ is a symmetric $\alpha$-stable-like process associated with the operator $\hat L$ on $L^2(\R^d;\hat \mu)$ and $\hat\mu(dx)=\hat a(x)^{-1}\,dx.$ Since $\hat a$ is bounded from above and below, $\hat \mu$ is comparable to the Lebesgue measure. Let $\hat p^{\hat \mu}_{B(x_0,n)}(t,x_1,x_2)$ be the Dirichlet heat kernel with respect to $\hat \mu$ of the process $\hat X$ killed on exiting from $B(x_0,n)$. It follows from \cite[Theorem 1.1]{CKS} or \cite[Theorem 1.2]{KK} that
$$\hat\Pp^{x_0}(\hat X_t\in A,\tau_{B(x_0,n)}^{\hat X}>t)=\int_A \hat p^{\hat \mu}_{B(x_0,n)}(t,x_0,y)\,\hat \mu(dy)>0.$$ Thus, $\Pp^{x_0}(X_t^\mu\in A)>0$; that is, the semigroup $(P_t)_{t\ge0}$ is irreducible. Hence, the claims that the first eigenvalue $\lambda_1$ is one-dimensional and the existence of a strictly positive version of the associated first eigenfunction are a consequence of \cite[Chapter V, Theorem 6.6]{Sch}.

Furthermore, according to part (1) in the proof of Theorem \ref{T:com1}, $(P_t)_{t\ge0}$ is ultracontractive. So,
$\phi(x)=e^{\lambda_1 t} P_t\phi(x)\le e^{\lambda_1 t}\|P_t\|_{L^2(\R^d;\mu)\to L^\infty (\R^d,\mu)}<\infty.$ This proves that $\phi$ is bounded. Moreover, let $p_\mu(t,x,y)$ be the heat kernel of $(P_t)_{t\ge0}$ with respect to the measure $\mu$. Also by part (1) in the proof of Theorem \ref{T:com1},
there is a sequence of compact sets $(F_k)_{k\ge1}$ such that $\cup_{k=1}^\infty F_k=\R^d$ and that for every $t>0$ and $y\in \R^d$, $x\mapsto p_\mu(t,x,y)$ is continuous on each $F_k$. Then, by the facts that
$$\phi(x)=e^{\lambda_1 t} P_t\phi(x)=e^{\lambda_1 t}\int_{\R^d} p_\mu(t,x,y)\phi(y)\,\mu(dy),\quad x\in \R^d,t>0 $$ and $\phi$ is strictly positive, we can get that for all compact sets $A$, $\inf_{x\in A}\phi(x)>0.$
 \end{proof}

In the following, we always assume that $\phi$ is strictly positive and bounded such that for all compact sets $A$, $\inf_{x\in A}\phi(x)>0.$ We have the following lower bound estimate for $\phi$.
\begin{lemma}\label{l:low}Assume that $d>\alpha$. Then there exists a constant $c>0$ such that for all $x\in \R^d$,
$$\phi(x)\ge c(1+|x|)^{\alpha-d}.$$ \end{lemma}
\begin{proof}
Note that for all $x\in \R^d$,
$$(-\Delta)^{\alpha/2}\phi(x)=\frac{\lambda_1 \phi(x)}{W(x)}.$$
Then, for $d>\alpha$, due to the fact that $\inf_{|x|\le 1}\phi(x)>0$,
\begin{align*}
\phi(x)=&c_{d,\alpha}\lambda_1\int_{\R^d} \frac{ \phi(y)}{W(y) |x-y|^{d-\alpha}}\,dy\ge c_1\int_{\{|y|\le 1\}} \frac{1}{|x-y|^{d-\alpha}}\,dy\ge c_2(1+|x|)^{\alpha-d}.
\end{align*} The proof is completed.
\end{proof}

We also need upper bounds for $\phi$.
\begin{lemma}\label{L:up} Assume that $d>\alpha$. Then, there exists a constant $c>0$ such that for all $x\in \R^d$,
\begin{equation}\label{p:el0}\phi(x)\le c (1+|x|)^{\alpha-d} \left(1+\I_{\big\{\frac{d-\beta}{\beta-\alpha}\in \Z_+\big\}}\log (2+|x|)\right),\end{equation} where $\Z_+:=\{0,1,2,\cdots\}.$ \end{lemma}
\begin{proof} The proof is based on the following estimates obtained in \cite[Lemma 6.1]{MS12-2}. Let $\gamma, \beta>0$ such that $\gamma<d$ and $\gamma+\beta>d$. Set $$J(x)=\int_{\R^d}\frac{1}{|x-y|^\gamma(1+|y|)^\beta}\,dy.$$ Then, there is a constant $c_1>0$ such that
\begin{equation}\label{p:el1}
J(x)\le
\begin{cases}
c_1(1+|x|)^{d-(\gamma+\beta)}, &\text{if}\ \beta<d,\\
c_1(1+|x|)^{-\gamma}\log(2+|x|), &\text{if}\ \beta=d,\\
c_1(1+|x|)^{-\gamma},&\text{if}\ \beta>d.
\end{cases}
\end{equation}

As in the proof of Lemma \ref{l:low}, we have
\begin{equation}\label{e:phi}\phi(x)=c_{d,\alpha}\lambda_1\int_{\R^d} \frac{\phi(y)}{|x-y|^{d-\alpha}W(y)}\,dy.\end{equation}Since, by Lemma \ref{L:re}, $\phi$ is strictly positive and bounded, the required assertion for the case that $\beta\ge d$ immediately follows from \eqref{p:el1} and the fact that $W(x)\ge c_0(1+|x|)^\beta$ as indicated in \eqref{E:11}. When $\beta<d$, we first get from \eqref{p:el1} and the fact $W(x)\ge c_0(1+|x|)^\beta$ that $$\phi(x)\le c_1 (1+|x|)^{-\beta+\alpha},\quad x\in \R^d.$$ Then, according to \eqref{e:phi} and also the fact that $W(x)\ge c_0(1+|x|)^\beta$, we deduce
\begin{align*}\phi(x)&\le c_2\int_{\R^d} \frac{1}{|x-y|^{d-\alpha}(1+|y|)^{\beta+(\beta-\alpha)}}\,dy.\end{align*} If $\beta+(\beta-\alpha)\ge d$, then the desired assertion follows from \eqref{p:el1} again; otherwise, we iterate this argument $k$ times, $k$ being the smallest integer such that $\beta+k(\beta-\alpha)\ge d$, and then can obtain the desired conclusion, also thanks to \eqref{p:el1}.\end{proof}

\begin{proposition}\label{P:nash}Assume that $d>\alpha$. Then, there exists a constant $c_1>0$ such that for any $u\in C_c^\infty(\R^d)$,
$$\left(\int_{\R^d} u^2\,d\mu\right)^{(d+\alpha)/d}\le c_1 D(u,u)\left(\int_{\R^d}|u|\phi^{({\beta/(2\alpha))\wedge1}}\,d\mu\right)^{2\alpha/d}.$$
\end{proposition}

\begin{proof} By the H\"{o}lder inequality, we find that for any $u\in C_c^\infty(\R^d)$,
\begin{equation}\label{e:bound}\begin{split} \int_{\R^d}|u|^2\,d\mu=&\int_{\R^d} |u(x)|^2W(x)^{-1}\,dx\\
=&\int_{\R^d}|u(x)|^{2d/(d+\alpha)}|u(x)|^{2\alpha/(d+\alpha)}W(x)^{-1}\,dx\\
\le &\left( \int_{\R^d}|u(x)|^{2d/(d-\alpha)}\,dx\right)^{(d-\alpha)/(d+\alpha)}\\
&\times \left(\int_{\R^d} |u(x)| W(x)^{-(d+\alpha)/(2\alpha)}\,dx\right)^{2\alpha/(d+\alpha)}\\
= &\left( \int_{\R^d}|u(x)|^{2d/(d-\alpha)}\,dx\right)^{(d-\alpha)/(d+\alpha)}\\
&\times \left(\int_{\R^d} |u(x) | W(x)^{-(d-\alpha)/(2\alpha)}\,\mu(dx)\right)^{2\alpha/(d+\alpha)}\\
\le& cD(u,u)^{d/(d+\alpha)}\left(\int_{\R^d} |u(x) | (1+|x|)^{-\beta(d-\alpha)/(2\alpha)}\,\mu(dx)\right)^{2\alpha/(d+\alpha)},
\end{split}\end{equation} where the last inequality follows from the fractional Sobolev inequality and the fact that $W(x)\ge c_0(1+|x|)^\beta$. The inequality above along with Lemma \ref{l:low} and $d>\alpha$ yields the desired assertion. \end{proof}

The following statement follows from Proposition \ref{P:nash}, Lemma \ref{L:up} and Theorem \ref{T:w1}.

\begin{proposition}\label{P:hk} Assume that $d>\alpha$. Then $(P_t)_{t\ge0}$ has a kernel $p_\mu(t,x,y)$ with respect to $\mu$ such that there exists a constant $c_1:=c_1(\alpha,\beta)>0$ such that for any $x,y\in \R^d$ and $t>0$,
$$p_\mu(t,x,y)\le c_1 t^{-d/\alpha}\phi^{(\beta/(2\alpha)\wedge1)}(x)\phi^{(\beta/(2\alpha)\wedge1)}(y).$$
\end{proposition}
\begin{proof}We only prove the case that $\alpha<\beta\le 2\alpha$, since the case that $\beta\ge 2 \alpha$ can be verified similarly. Taking $V=\phi^{\beta/(2\alpha)}$, we get from Lemma \ref{L:up} and the fact that $W(x)\ge c_0(1+|x|)^\beta$ that
\begin{equation}\label{e:upper}\begin{split}\mu(V^2)&=\mu(\phi^{\beta/\alpha})\\
&\le c_0\int_{\R^d} (1+|x|)^{\beta(\alpha-d)/\alpha} (\log^{\beta/\alpha}(2+|x|)) (1+|x|)^{-\beta}\,dx<\infty,\end{split}\end{equation} due to $\beta>\alpha$. On the other hand, since $P_t\phi=e^{-\lambda_1t}\phi\le \phi$ for all $t>0$, by the Jensen inequality and the fact that $\alpha<\beta\le 2\alpha$, we have
$$P_t V=P_t \phi^{\beta/(2\alpha)}\le (P_t\phi)^{\beta/(2\alpha)}\le \phi^{\beta/(2\alpha)}=V,\quad t>0.$$ Therefore, $V$ is a Lyapunov function with constant $c=0$ in the sense of Definition \ref{D:1}.

Furthermore, by Proposition \ref{P:nash}, we know that weighted Nash inequalities hold with such weighted function $V$ and the rate function $\psi(r)=r^{(d+\alpha)/d}.$ Therefore, the desired assertion is a direct consequence of Theorem \ref{T:w1}.\end{proof}

Now, it is a position to present the main result in this section.
\begin{theorem}\label{T:main}Suppose that $d>\alpha$ and $L=-W(x)(-\Delta)^{\alpha/2}$, where $W(x)$ is a weighted function satisfying \eqref{E:11} with $\beta>\alpha$ and $\delta>1$. Then the following two statements hold.
\begin{itemize}
\item[(i)] There exists a constant $c_1\ge1$ such that for all $n\ge1$,
$$c_1^{-1}n^{\alpha/d}\le \lambda_n\le c_1n^{\alpha/d}.$$
\item[(ii)] There exists a constant $c_2\ge1$ such that for all $t\in (0,1]$,
$$c_2^{-1} t^{-d/\alpha}\le \sup_{x,y\in\R^d} \frac{p_\mu(t,x,y)}{\phi^{(\beta/(2\alpha))\wedge1}(x)\phi^{(\beta/(2\alpha))\wedge1}(y)}\le c_2t^{-d/\alpha}.$$
\end{itemize}
 \end{theorem}

\begin{proof} (1) According to Proposition \ref{P:hk}, $\phi\in L^2(\R^d;\mu)$ and \eqref{e:upper}, the density function $p_\mu(t,x,y)$ satisfies that for all $t>0$,
\begin{equation}\label{e:jup1}\int_{\R^d} p_\mu(t,x,x)\,\mu(dx)\le c_0t^{-d/\alpha}.\end{equation}

Next, we use the comparison approach again as in the proof of Lemma \ref{L:re}. For any $x\in \R^d$ with $|x|\le1$, $0<r\le 1$ and $t>0$,
\begin{align*}\Pp^x(X_t^\mu\in B(x,r))&\ge \Pp^x(X_t^\mu\in B(x,r), \tau_{B(0,2)}^{X^\mu}>t)\\
&=\hat\Pp^x(\hat X_t\in B(x,r),\tau_{B(0,2)}^{\hat X}>t).\end{align*} Here $\hat X=(\hat X_t)_{t\ge0}$ is a symmetric $\alpha$-stable-like process generated by the operator $\hat L=-\hat a (-\Delta)^{\alpha/2}$ on $L^2(\R^d;\hat \mu)$, where $\hat\mu(dx)=\hat a(x)^{-1}\,dx$ and $0<a_1\le \hat a(x)\le a_2<\infty$ for all $x\in \R^d$ such that $\hat a(x)=W(x)$ for all $x\in B(0,2)$. Let $\hat p^{\hat \mu}_{B(0,2)}(t,x,y)$ be the Dirichlet heat kernel with respect to $\hat\mu$ of the process $\hat X$ killed on exiting from $B(0,2)$. In particular, it holds that for all $x\in \R^d$ with $|x|\le 1$ and $0<t\le 1$,
\begin{equation}\label{e:llow}p_\mu(t,x,x)\ge c_1 \hat p^{\hat \mu}_{B(0,2)}(t,x,x)\ge c_2 t^{-d/\alpha},\end{equation} where we have used \cite[Theorem 1.1]{CKS} or \cite[Theorem 1.2]{KK} in the last inequality.
Therefore, for all $0<t\le 1$,
\begin{equation}\label{e:jup2}\int_{\R^d} p_\mu(t,x,x)\,\mu(dx)\ge \int_{B(0,2)} p_\mu(t,x,x)\,\mu(dx)\ge c_3 t^{-d/\alpha}.\end{equation} Combining \eqref{e:jup1} and \eqref{e:jup2} with Proposition \ref{P:ei1}, we can obtain the desired bounds for $\lambda_n.$

(2) The upper bound of the required assertion immediately follows from Proposition \ref{P:hk}. Note that
$$\sup_{x,y\in\R^d} \frac{p_\mu(t,x,y)}{\phi^{(\beta/(2\alpha))\wedge1}(x)\phi^{(\beta/(2\alpha))\wedge1}(y)}\ge \sup_{x\in \R^d:|x|\le 1} \frac{p_\mu(t,x,x)}{\phi^{(\beta/\alpha)\wedge1}(x)}\ge c_1  \sup_{x\in \R^d: |x|\le 1} p_\mu(t,x,x),$$ where in the last inequality we have used \eqref{p:el0}. This along with \eqref{e:llow} yields that
$$\sup_{x,y\in\R^d} \frac{p_\mu(t,x,y)}{\phi^{(\beta/(2\alpha))\wedge1}(x)\phi^{(\beta/(2\alpha))\wedge1}(y)}\ge c_2t^{-d/\alpha}.$$ The proof is finished.
\end{proof}

\subsection{Further properties for weighted fractional semigroups  generated by $L_0=-(1+|x|)^\beta(-\Delta)^{\alpha/2}$}
From this subsection, we will concentrate on {\it the operator  $L_0=-(1+|x|)^\beta(-\Delta)^{\alpha/2}$ with $d\wedge \beta>\alpha$}. In particular, all  results in the previous subsection hold true for $L_0$.

The following statement can be seen as a complementary of Proposition \ref{P:hk} for the operator $L_0$ in case that $\alpha<\beta< 2\alpha$.
\begin{proposition}\label{P:hkc}Assume that $d>\alpha$ and $\alpha<\beta<2\alpha$. Then $(P_t)_{t\ge0}$ has a kernel $p_\mu(t,x,y)$ with respect to $\mu$ such that for any $x,y\in \R^d$ and $t>0$,
$$p_\mu(t,x,y)\le c_3 t^{-(d+\beta-2\alpha)/(\beta-\alpha)}\phi(x)\phi(y)$$ for some constant $c_3:=c_3(\alpha,\beta)>0$.
\end{proposition}

\begin{proof} Similar to the proof of Proposition \ref{P:nash}, we make use of weighted Nash inequalities. Set $\theta=\frac{\alpha(d+\beta-2\alpha)}{\beta-\alpha}>\alpha$. By the H\"{o}lder inequality, for all $u\in C_c^\infty(\R^d)$,
\begin{align*}\int_{\R^d}u^2\,d\mu&=\int_{\R^d} |u|^{\frac{2\theta}{\theta+\alpha}}(x)(1+|x|)^{-\frac{2\alpha(\alpha-d)}{\theta+\alpha}} \cdot|u|^{\frac{2\alpha}{\theta+\alpha}}(x)(1+|x|)^{\frac{2\alpha(\alpha-d)}{\theta+\alpha}}
\,\mu(dx)\\
&\le\left(\int_{\R^d} |u|^{\frac{2\theta}{\theta-\alpha}}(x)(1+|x|)^{-\frac{2\alpha(\alpha-d)}{\theta-\alpha}} \,\mu(dx)\right)^{\frac{\theta-\alpha}{\theta+\alpha}}\\
&\quad \times \left(\int_{\R^d}|u|(x)(1+|x|)^{\alpha-d}\,\mu(dx)\right)^{\frac{2\alpha}{\theta+\alpha}}\\
&=\left(\int_{\R^d} |u|^{\frac{2\theta}{\theta-\alpha}}(x)(1+|x|)^{-\frac{2\alpha(\alpha-d)}{\theta-\alpha}-\beta} \,dx\right)^{\frac{\theta-\alpha}{\theta+\alpha}}\\
&\quad \times \left(\int_{\R^d}|u|(x)(1+|x|)^{\alpha-d}\,\mu(dx)\right)^{\frac{2\alpha}{\theta+\alpha}}. \end{align*}
Applying Proposition \ref{CKN} with $$\tau=\frac{2\theta}{\theta-\alpha}=\frac{2(d+\beta-2\alpha)}{d-\alpha}>0$$ and $$\gamma=\frac{(\beta-2\alpha)(d-\alpha)}{2(d+\beta-2\alpha)},$$ and noticing that
$$ \gamma\tau=\beta-2\alpha=-\frac{2\alpha(\alpha-d)}{\theta-\alpha}-\beta,$$ we obtain that
$$ \left(\int_{\R^d} |u|^{\frac{2\theta}{\theta-\alpha}}(x)(1+|x|)^{-\frac{2\alpha(\alpha-d)}{\theta-\alpha}-\beta} \,dx\right)^{\frac{\theta-\alpha}{\theta}}\le \||x|^\gamma u\|^2_{L^\tau(\R^d;dx)}\le c_1D(u,u)$$ holds for some constant $c_1>0$, which is independent of $u$. Thus, combining both estimates above with Lemma \ref{l:low}, we arrive at that for all $u\in C_c^\infty(\R^d)$,
\begin{align}\label{e:ffee1}\left(\int_{\R^d}u^2\,d\mu\right)^{(\theta+\alpha)/\theta}\le& c_2 D(u,u)\left(\int_{\R^d}|u|(x)(1+|x|)^{\alpha-d}\,\mu(dx)\right)^{2\alpha/\theta}\\
\le&c_3 D(u,u)\left(\int_{\R^d}|u|(x)\phi(x)\,\mu(dx)\right)^{2\alpha/\theta}.\nonumber\end{align}  Therefore, the required assertion follows from the inequality above and Theorem \ref{T:w1}.\end{proof}
\begin{corollary}\label{e:Cor}Assume that $d\wedge\beta>\alpha$. Then there exists a constant $c_1\ge1$ such that for all $x\in \R^d$,
$$c_1^{-1}(1+|x|)^{\alpha-d}\le \phi(x)\le c_1(1+|x|)^{\alpha-d}.$$\end{corollary}
\begin{proof} The lower bound has been proven in Lemma \ref{l:low}, and so we need to verify the upper bound. According to Theorem \ref{T:main}(ii) and Proposition \ref{P:hkc}, we know that for any $t>0$, there exists a constant $c(t)>0$ such that for all $x,y\in \R^d$,
$$p_\mu(t,x,y)\le c(t) \phi(x)\phi(y).$$ Therefore, the semigroup $(P_t)_{t\ge0}$  is
intrinsically ultracontractive, see \cite[Section 3]{DS}. Furthermore, in view of \cite[Theorem 3.2]{DS}, we know that for any $t>0$, there exists a constant $c'(t)>0$ such that for all $x,y\in \R^d$,
$$c'(t)\phi(x)\phi(y)\le p_\mu(t,x,y).$$

On the other hand, by \eqref{e:ffee1} and \eqref{e:ffee1}, we can see that the so called Nash inequalities indeed hold with the weighted function $V(x)=(1+|x|)^{\alpha-d}$. Then, according to Theorem \ref{T:w1}, Propositions \ref{P:hk} and \ref{P:hkc} still hold true with $\phi$ replaced by $(1+|x|)^{\alpha-d}$. In particular, for any $t>0$, there exists a constant $c''(t)>0$ such that for all $x,y\in \R^d$,
$$p_\mu(t,x,y)\le c''(t) (1+|x|)^{\alpha-d}(1+|y|)^{\alpha-d}.$$

Combining with both conclusions above, we can get that there exists a constant $c_1>0$ such that for all $x\in \R^d$,
$$\phi(x)\le c_1(1+|x|)^{\alpha-d}.$$ The proof is complete.
\end{proof}

\begin{remark} According to Theorem \ref{T:main}(ii), for $0<t\le 1$, the factor $t^{-d/\alpha}$  among the estimates of the form $c(t)\phi(x)\phi(y)$ is optimal when $\beta\ge 2\alpha$; namely, $d/\alpha$ can not been replaced by any positive constant strictly smaller than $d/\alpha$. On the other hand, we can show that, concerning the estimates in Proposition \ref{P:hkc}, the factor $t^{-(d+\beta-2\alpha)/(\beta-\alpha)}$ is also optimal when $\alpha<\beta<2\alpha$. Indeed, by Corollary \ref{e:Cor}, there is a constant $c_1\ge1$ such that for $x\in\R^d$,
$$c_1^{-1}(1+|x|)^{\alpha-d}\le \phi(x)\le c_1(1+|x|)^{\alpha-d}.$$ Assume that
$$p_\mu(t,x,y)\le c_0t^{-\gamma}(1+|x|)^{\alpha-d}(1+|y|)^{\alpha-d},\quad t>0,x,y\in\R^d$$ holds for some $\gamma>0$. According to Theorem \ref{T:w2}, the following weighted Nash inequality
$$\left(\int_{\R^d}u^2\,d\mu\right)^{1+(1/\gamma)}\le c_1D(u,u)\left(\int_{\R^d} |u|(x)(1+|x|)^{\alpha-d}\,d\mu\right)^{2/\gamma}$$ holds for all $u\in C_c^\infty(\R^d)$. By replacing $u(x)$ with $u(\lambda x)$ for all $\lambda>0$ in the inequality above, we obtain
\begin{align*}\lambda^{\beta-\alpha-\frac{d+\beta-2\alpha}{\gamma}}&\left(\int_{\R^d}\frac{u^2(x)}{(\lambda+|x|)^\beta}\,dx\right)^{1+(1/\gamma)}\le c_1 D(u,u)\left(\int_{\R^d}\frac{|u|(x)}{(\lambda+|x|)^{\beta+d-\alpha}}\,dx\right)^{2/\gamma}.\end{align*} Note that the constant $c_1$ is independent of $\lambda$. Letting $\lambda\to 0$, the inequality above can be true only if $$\beta-\alpha-\frac{d+\beta-2\alpha}{\gamma}\ge0;$$ that is,
$$\gamma\ge \frac{d+\beta-2\alpha}{\beta-\alpha}.$$ We also note that
$$\frac{d+\beta-2\alpha}{\beta-\alpha}>\frac{d}{\alpha}$$ if and only if $2\alpha>\beta$.\end{remark}

\ \

Finally, we present a conclusion remark.

\begin{remark} This paper is concerned with weighted fractional Laplacian operators of the form $L=-W(x)(-\Delta)^{\alpha/2}$ on $\R^d$, where $\alpha\in(0,2)$ and $d>\alpha$. Part arguments can be extended to some cases beyond $d>\alpha$. For example, Theorem \ref{T:com1}(i) holds true for all $\alpha\neq d$, if we apply the Hardy inequality for the fractional Laplacian on $\R^d$ with $\alpha\neq d$, see also \cite[Theorem 1.1]{D04}. However, proofs of key statements, including Theorem \ref{T:com1}(ii), Lemmas \ref{l:low} and \ref{L:up}, are heavily based on the framework of $d>\alpha$. In particular, in Lemmas \ref{l:low} and \ref{L:up} we used Green function estimates for fractional Laplacian operators on $\R^d$, which are only available when $d>\alpha$. Motivated by results for the Laplacian operator with unbounded diffusion coefficient in $\R$ and $\R^2$ developed in \cite{MS13}, we believe that the corresponding analytic properties and their approaches for fractional heat semigroups in the case that $d\le \alpha$ should be
different from those in the present setting. We will discuss these issues in future work. \end{remark}

\section{Appendix}
\subsection{Weighted Nash inequalities}
The main tool to derive upper bounds of heat kernel  is the following weighted Nash inequalities, which are originally due to F.-Y. Wang, see \cite[Theorem 3.3]{Wang02}. We recall statements from \cite{BBGM} adopted to our situation. Let $(P_t)_{t\ge0}$ be a symmetric Markov semigroup on $L^2(\R^d;\mu)$ with generator $L$.
The following definition is taken from \cite[Section 2.1]{MS12}.

\begin{definition}\label{D:1}\it A Lyapunov function $V$ is a positive function on $\R^d$ such that $V\in L^2(\R^d;\mu)$ and
$$P_tV(x)\le e^{ct}V(x),\quad x\in \R^d, t>0$$ for some constant $c\ge0$. Here, $c$ is called the Lyapunov constant. \end{definition}

Following \cite[Remark 1]{MS12}, we give some remark for Definition \ref{D:1}. In \cite[Definition 3.3]{BBGM}, the Lyapunov function $V$ is required to be in the domain of the operator $L$ such that $LV\le cV$. Such condition is not so easy to verify in applications. Indeed, according to the proof of \cite[Theorem 3.5]{BBGM} (see Theorem \ref{T:w1} below), such assumption is only used to ensure the validity of \cite[(3.5)]{BBGM} for all positive functions $f\in L^2(\R^d;\mu)$. Observe that, if $V$ is a Lyapunov function in the sense of Definition \ref{D:1}, then, for all positive functions $f\in L^2(\R^d;\mu)$,
\begin{align*}\int_{\R^d} VP_t f\,d\mu=\int_{\R^d} f P_tV\,d\mu\le e^{ct}\int_{\R^d} fV\,d\mu.\end{align*} That is, \cite[(3.5)]{BBGM} holds true.

\begin{definition}(\cite[Definition 3.2]{BBGM}) \it Let $V$ be a positive function on $\R^d$ and $\psi$ be a positive function defined on $(0,\infty)$ such that $r\mapsto\frac{\psi(r)}{r}$ is increasing on $(0,\infty)$. A Dirichlet form $(D,\mathscr{D}(D))$ on $L^2(\R^d;\mu)$ satisfies a weighted Nash inequality with weight $V$ and rate function $\psi$, if
$$\psi\left(\frac{ \|u \|^2_{L^2(\R^d;\mu)}}{\,\,\,\|u V\|^2_{L^1(\R^d;\mu)}}\right)\le \frac{D(f,f)}{\quad\,\,\|u V\|^2_{L^1(\R^d;\mu)}}$$ for all functions $u \in \mathscr{D}(D)$ such that $\|u \|_{L^2(\R^d;\mu)}>0$ and $\|u V\|_{L^1(\R^d;\mu)}<\infty.$ \end{definition}

\begin{theorem}\label{T:w1}$($\cite[Theorem 3.5 and Corollary 3.7]{BBGM}$)$
Assume that there exists a Lyapunov function $V$ in $L^2(\R^d;\mu)$ with Lyapunov constant $c\ge0$,
and that the Dirichlet form $(D,\mathscr{D}(D))$ associated to $L$ satisfies a weighted Nash
inequality with weight $V$ and rate function $\psi$ on $(0,\infty)$ such that
$$\int_1^\infty \frac{1}{\psi(r)}\,dr<\infty\quad \mathrm{and   }\,\,\,\, \int_0^1 \frac{1}{\psi(r)}\,dr=\infty.$$
Then
$$\|P_tf\|^2_{L^2(\R^d;\mu)}\le \Psi^{-1}(2t) e^{2ct}\|fV\|^2_{L^1(\R^d;\mu)}$$ for all $t>0$ and all functions $f\in L^2(\R^d;\mu)$. Here, the function $\Psi$ is defined by
$$\Psi(t)=\int_t^\infty \frac{1}{\psi(r)}\,dr$$ and $\Psi^{-1}(t)=\inf\{r>0: \Phi(r)\le t\}.$ In particular, $(P_t)_{t\ge0}$ has a kernel $p_\mu(t,x,y)$ with respect to $\mu$ which satisfies
$$p_\mu(t,x,y)\le \Psi^{-1}(t)e^{ct}V(x)V(y)$$ for all $t>0$ and $x,y\in \R^d\times \R^d$.  \end{theorem}

Theorem  \ref{T:w1} has the following converse.
\begin{theorem}\label{T:w2}$($\cite[Theorem 3.9 and Proposition 3.1]{BBGM}$)$ Let $(P_t)_{t\ge0}$ be a symmetric Markov semigroup on $L^2(\R^d;\mu)$ which has a kernel $p_\mu(t,x,y)$ with respect to $\mu$ such that
$$p_\mu(t,x,y)\le \Psi(t)V(x)V(y)$$ for all $t>0$ and $x,y\in \R^d\times \R^d$, where $V$ is a positive function on $\R^d$ and $\Psi$ is a positive function on $(0,\infty)$. Then, the weighted Nash inequality holds with this function $V$ and rate function $\psi$ as follows:
$$\psi(r)=\sup_{t>0}\frac{r}{t}\log \frac{r}{\Psi(t)}.$$
 \end{theorem}

\subsection{Heat kernel estimates and eigenvalues}
Let $0<\lambda_1<\lambda_2\le \cdots\lambda_n\le \cdots$ be eigenvalues of $-L$, repeated according to their multiplicity, and $\phi_n$ be the corresponding eigenfunctions, which we assume to be normalized in $L^2(\R^d;\mu)$. Then, for any $f\in L^2(\R^d;\mu)$,
$$P_tf(x)=\sum_{n=1}^\infty c_ne^{-\lambda_nt}\phi_n(x),\quad t>0,x\in \R^d,$$ where $c_n=\int_{\R^d} f(z)\phi_n(z)\,\mu(dz).$ Therefore, for any $t>0$ and $x,y\in \R^d$,
\begin{align*}p_\mu(t,x,y)&=\int_{\R^d} p_\mu(t/2,x,z)p_\mu(t/2,y,z)\,\mu(dz)=P_{t/2} p(t/2,y,\cdot)(x)\\
&=\sum_{n=1}^\infty e^{-\lambda_n t/2}\phi_n(x)\int_{\R^d} p_\mu(t/2,y,z)\phi_n(z)\,\mu(dz)\\
&=\sum_{n=1}^\infty e^{-\lambda_n t/2}\phi_n(x)P_{t/2}\phi_n(y)=\sum_{n=1}^\infty e^{-\lambda_nt}\phi_n(x)\phi_n(y),\end{align*} where in the last equality we used the fact that $$P_t\phi_n(x)=e^{-\lambda_n t}\phi_n(x),\quad t>0,x\in\R^d.$$ In particular, it holds that for all $t>0$,
\begin{equation}\label{e:eigen}\int_{\R^d} p_\mu(t,x,x)\,\mu(dx)=\sum_{n=1}^\infty e^{-\lambda_nt}.\end{equation}
Note that \eqref{e:eigen} is well known in the spectral theory, and has been used to estimate higher order eigenvalues, see e.g.\ \cite[(2.1)]{CP}, \cite[Lemma 3.3.18 and Theorem 3.3.19]{WBook} or \cite[Proposition 4.1]{MS12-2}.

According to \eqref{e:eigen} and \cite[Proposition 7.2]{MS12-2}, we have

\begin{proposition}\label{P:ei1}Suppose that there are constants $\theta>0$ and $C_1\ge1$ such that
$$C_1^{-1}\le \liminf_{t\to0} \frac{\int_{\R^d} p_\mu(t,x,x)\,\mu(dx)}{t^{-\theta}}\le \limsup_{t\to0} \frac{\int_{\R^d} p_\mu(t,x,x)\,\mu(dx)}{t^{-\theta}}\le C_1.$$ Then,
$$C_2^{-1}\le \liminf_{n\to\infty}\frac{\lambda_n}{n^{1/\theta}}\le \limsup_{n\to\infty}\frac{\lambda_n}{n^{1/\theta}}\le C_2$$ for some constant $C_2\ge1$.\end{proposition}

\subsection{Fractional Caffarelli-Kohn-Nirenberg inequalities}
The following statement follows from \cite[Theorem 1.1, i)]{NS}.
\begin{theorem}\label{CKN} Let $d\ge1$, $\tau>0$ and $\gamma\in \R$ such that $$\frac{1}{\tau}+\frac{\gamma}{d}=\frac{1}{2}-\frac{\alpha}{2d}.$$ Then there exists a constant $C>0$ such that for all $u\in C_c^1(\R^d)$,
$$\||x|^\gamma u\|^2_{L^\tau(\R^d;dx)}\le C D(u,u).$$ \end{theorem}

\noindent{\bf Acknowledgements.} The author would like to think the referee for helpful comments and careful corrections. The research is supported by National
Natural Science Foundation of China (No.\ 11522106), the Fok Ying Tung
Education Foundation (No.\ 151002) and the Program for Probability and Statistics: Theory and Application (No. IRTL1704).

\end{document}